\documentclass[11pt]{amsart}
\usepackage{geometry}
\usepackage{enumerate}
\usepackage{amssymb}
\usepackage{color}
\let\Oldepsilon\epsilon
\let\Oldvarepsilon\varepsilon
  \renewcommand{\varepsilon}{\Oldepsilon}
  \renewcommand{\epsilon}{\Oldvarepsilon}
\let\Oldphi\phi
\let\Oldvarphi\varphi
  \renewcommand{\varphi}{\Oldphi}
  \renewcommand{\phi}{\Oldvarphi}

\newcommand{\C}{\mathbb{C}}
\newcommand{\R}{\mathbb{R}}

\newcommand{\Z}{\mathbb{Z}}

\newcommand{\cF}{\mathcal{F}}

\theoremstyle{plain}
\newtheorem{theorem}{Theorem}[section]
\newtheorem{corollary}[theorem]{Corollary}
\newtheorem{lemma}[theorem]{Lemma}
\newtheorem{proposition}[theorem]{Proposition}
\newtheorem*{maintheorem}{Main Theorem}

\theoremstyle{definition}

\theoremstyle{remark}

\numberwithin{equation}{section}

\newcommand{\lip}{\left\langle}
\newcommand{\rip}{\right\rangle}

\newcommand{\lnorm}{\left\Vert}
\newcommand{\rnorm}{\right\Vert}
\newcommand{\biglnorm}{\bigl\|}
\newcommand{\bigrnorm}{\bigr\|}

\newcommand{\lset}{\left\lbrace}
\newcommand{\rset}{\right\rbrace}

\newcommand{\labs}{\left\vert}
\newcommand{\rabs}{\right\vert}
\newcommand{\biglabs}{\bigl\vert}
\newcommand{\bigrabs}{\bigr\vert}

\newcommand{\lpar}{\left(}
\newcommand{\rpar}{\right)}
\newcommand{\biglpar}{\bigl(}
\newcommand{\bigrpar}{\bigr)}
\newcommand{\Biglpar}{\Bigl(}
\newcommand{\Bigrpar}{\Bigr)}

\newcommand{\hatotimes}{\mathbin{\hat{\otimes}}}
\newcommand{\group}[1]{\mathrm{#1}}
\newcommand{\half}{\frac{1}{2}}

\newcommand{\Hilb}{\mathcal{H}}
\newcommand{\vect}{\mathcal{V}}
\newcommand{\ub}{\mathrm{ub}}
\newcommand{\cb}{\mathrm{cb}}
\newcommand{\op}{\mathrm{op}}

\newcommand{\cosec}{\operatorname{cosec}}
\newcommand{\sgn}{\operatorname{sgn}}
\renewcommand{\Re}{\operatorname{Re}}
\renewcommand{\Im}{\operatorname{Im}}

\newcommand{\dvec}{}

\allowdisplaybreaks

\begin{document}
\title[Completely bounded multipliers of $\group{SL}(2,\R)$]{Uniformly bounded representations \\
and completely bounded multipliers of $\group{SL}(2,\R)$}
\author{Francesca Astengo}
\address
{Dipartimento di Matematica, Universit\`a di Genova,
16146 Genova, Italia}
\email{astengo@dima.unige.it}

\author{Michael G. Cowling}
\address
{School of Mathematics and Statistics, University of New South Wales,
UNSW Sydney 2052, Australia}
\email{m.cowling@unsw.edu.au}

\author{Bianca Di Blasio}
\address
{Dipartimento di Matematica e Applicazioni\\
Universit\`a di Milano Bicocca \\
Via Cozzi 53\\
20125 Milano\\ Italy}
\email{bianca.diblasio@unimib.it}


\begin{abstract}
We  estimate the norms of many matrix coefficients of irreducible uniformly bounded representations of $\group{SL}(2,\R)$ as completely bounded multipliers of the Fourier algebra.
Our results suggest that the known inequality relating the uniformly bounded norm of a representation and the completely bounded norm of its coefficients may not be optimal.
\end{abstract}

\maketitle
\section{Introduction}
We begin by summarising our results.
A representation $\pi$, by which we always mean a continuous representation of a locally compact group $G$ on a Hilbert space $\Hilb_\pi$, is said to be \emph{uniformly bounded} if $\pi(x)$ is a bounded operator on $\Hilb_\pi$ for each $x \in G$, and there is a constant $C$, necessarily no less than $1$, such that
\begin{equation}\label{eq:def-ub}
C^{-1}  \lnorm\dvec{v}\rnorm_{\Hilb_\pi}  \leq \lnorm \pi(x) \dvec{v} \rnorm_{\Hilb_\pi} \leq C \lnorm\dvec{v}\rnorm_{\Hilb_\pi}
\qquad\forall x \in G \quad\forall \dvec{v} \in \Hilb_\pi;
\end{equation}
the two inequalities are equivalent because $\pi$ is a representation.
We write $\lnorm  \pi(x)  \rnorm_{\op}$ for the operator norm of $\pi(x)$ and define the \emph{norm} of $\pi$, written $\lnorm  \pi \rnorm_{\ub}$, to be the smallest possible value of $C$ in this inequality.

Suppose that $\pi$ and $\sigma$ are uniformly bounded representations of $G$.
A linear operator from $\Hilb_\pi$ to $\Hilb_\sigma$ such that $\sigma(x) T = T \pi(x)$ for all $x \in G$ is called an intertwiner.
We say that $\pi$ and $\sigma$ are \emph{similar} if there is an intertwiner that is bounded with bounded inverse, and \emph{unitarily equivalent} if there is a unitary intertwiner.
Similarity and unitary equivalence are equivalence relations.
Similar uniformly bounded representations may have different norms, and hence not be equivalent.
In general, little seems to be known about similarity classes of uniformly bounded representations, or about finding uniformly bounded representations in an equivalence class with minimal norm.
Of course, if a unitary bounded representation is similar to a unitary representation, then the unitary representation has minimal norm in the equivalence class.

Around 1950, a number of researchers looked at uniformly bounded representations in their studies of amenability.
Once it was known that every uniformly bounded representation of an amenable group is unitarizable, that is, similar to a unitary representation, J.~Dixmier \cite{Dix} asked whether this was true in general or whether this characterized amenability.

In 1955, L.~Ehrenpreis and F.~Mautner~\cite{EhrMau1, EhrMau2} showed that $\group{SL}(2, \R)$ has two analytic families of representations $\pi_{\zeta, \epsilon}$, where $\zeta \in \C$ and $\epsilon$ is either $0$ or $1$.
These representations have bounded $K$-finite matrix coefficients  (here $K$ is $\group{SO}(2)$) if and only if $\labs \Re\zeta \rabs \leq \half$, and they are uniformly bounded when $\labs  \Re \zeta  \rabs< \half$; most of them are not similar to unitary representations.
One of the main techniques in the work of Ehrenpreis and Mautner is the use of generalised spherical functions $\phi_{\zeta, \epsilon}^{\mu,\nu}$, matrix coefficients of $\pi_{\zeta, \epsilon}$ that transform under the left and right actions of $K$ by imaginary exponentials, and are given on the diagonal subgroup of $\group{SL}(2,\R)$ by hypergeometric functions.
Harish-Chandra used the natural extensions of these generalised spherical functions in his studies of harmonic analysis on semisimple Lie groups.
These are known as generalised spherical functions as they extend the classical spherical functions, which transform trivially under the left and right actions of $K$.

Shortly after, R.A.~Kunze and E.M.~Stein \cite{KunStn1} found a use for these uniformly bounded representations, first realising them on the same Hilbert space, and then using them to prove what is now called the Kunze--Stein phenomenon for $\group{SL}(2,\R)$.
Since then, considerable effort has gone into the construction of uniformly bounded representations.
Apart from their fundamental paper \cite{KunStn1}, Kunze and Stein \cite{KunStn2, KunStn3, KunStn4}, as well as several other authors, constructed analytic families of uniformly bounded representations for many noncompact semisimple Lie groups in the 1960s and 1970s.
In the 1970s and 1980s, uniformly bounded representations were constructed for other groups; for example, A.~Fig\`a-Talamanca and M.A.~Picardello \cite{FigPic} and shortly after T.~Pytlik and R.~Szwarc \cite{PytSzw} found uniformly bounded representations of the free groups.
Very recently, K.~Juschenko and P.W.~Nowak \cite{JusNow} linked uniformly bounded representations with the exactness of discrete groups.

In the 1970s and 1980s, U.~Haagerup and his collaborators (see, for instance, \cite{Haa1, DcaHaa, CowHaa}) showed that completely bounded multipliers of the Fourier algebra have an important role to play in harmonic analysis.
As had already been remarked by J.E.~Gilbert \cite{Gil} and N.~Lohou\'e \cite{Loh}, each matrix coefficient of a uniformly bounded representation $\sigma$ of a group $G$, more precisely, for every choice of $\dvec{v},\dvec{w} \in \Hilb_\sigma$, the function $x \mapsto \lip \sigma(x) \dvec{v}, \dvec{w} \rip$,  which we abbreviate to $\lip \sigma \dvec{v}, \dvec{w} \rip$, is a completely bounded multiplier of the Fourier algebra $A(G)$, and there is a norm estimate
\begin{equation}\label{eq:cb-ub}
\lnorm  \lip \sigma \dvec{v}, \dvec{w} \rip  \rnorm_{\cb} \leq \lnorm  \sigma  \rnorm_{\ub}^2 \lnorm \dvec{v} \rnorm_{\Hilb_\sigma} \lnorm \dvec{w} \rnorm_{\Hilb_\sigma} .
\end{equation}
However, it was observed by Haagerup \cite{Haa2}, for the case in which $G = \group{SL}(2, \R)$ and $\lip \sigma \dvec{v}, \dvec{w} \rip$ is the spherical function associated to $\pi_{\zeta,\epsilon}$,
that this inequality is far from sharp.
We show here that similar equalities hold for the generalised spherical functions $\phi^{\mu,\nu}_{\zeta,\epsilon}$, which we define in \eqref{eq:def-gen-sph-functions} below.

\begin{maintheorem}
If $-\half < \Re\zeta < \half$ and $\mu \in 2\Z$, then
\[
\biglnorm \phi_{\zeta,0}^{\mu,\mu} \bigrnorm_{\cb}
\leq \sec(\pi\Re\zeta).
\]
More generally, if $0 < \xi_0 < \half$, $\epsilon \in \{0,1\}$ and $\mu, \nu \in 2\Z + \epsilon$, then
$\biglnorm \phi_{\zeta,\epsilon}^{\mu,\nu} \bigrnorm_{\cb}$ is uniformly bounded in the strip $\{ \zeta \in \C : \labs \Re\zeta \rabs \leq \xi_0\}$.
\end{maintheorem}

Our estimates are not uniform in $\mu$ and $\nu$, but this may be an artifact of our proof.
An (unpublished) announcement of V.~Losert gives us hope that this might be true.

In a preceding paper \cite{ACD}, we considered the Kunze--Stein representations $\pi_{\zeta, \epsilon}$; we proved that if $\sigma$ is a uniformly bounded representation of $\group{SL}(2, \R)$ that is similar to $\pi_{\zeta,\epsilon}$, where $\labs  \Re \zeta  \rabs < \half$ and $(\zeta, \epsilon) \neq (0,1)$, then
\begin{equation}\label{eq:unif-bounded-bound}
\lnorm   \sigma  \rnorm_{\ub}
\geq \lnorm  \pi_{\zeta, \epsilon}  \rnorm_{\ub}
\simeq \frac{(1+\labs \Im \zeta \rabs)^{\labs \Re \zeta \rabs}}{\half-\labs \Re \zeta  \rabs}
\end{equation}
when $ \Im\zeta $ is  large.
(Here and later in this paper, the expression $A(\zeta) \simeq B(\zeta) $ for all $\zeta \in E$, where $E$ is some subset of the domains of $A$ and of $B$, means that there exist constants $C$ and $C'$ such that $C\, A(\zeta)  \le B(\zeta)  \le C'\,A(\zeta)  $ for all $\zeta \in E$.
Our ``constants'' are all positive real numbers.)
This seems to confirm that the inequality \eqref{eq:cb-ub} is far from optimal, and suggests there may be a sharper version thereof yet to be unveiled, at least for the group $\group{SL}(2, \R)$.

We now provide more historical context for our results.
In the 1960s, N.Th.~Varopoulos \cite{Var1, Var2, Var3} used tensorial methods to answer questions about thin sets.
C.S.~Herz \cite{Her1, Her2}, inspired by Varopoulos, looked at pointwise multipliers of the projective tensor product $L^2(G) \mathbin{\hat{\otimes}} L^2(G)$, and connected these with noncommutative harmonic analysis, and in particular, with certain pointwise multipliers of the Fourier algebra that we will call completely bounded multipliers.
M.~Bo\.{z}ejko and G.~Fendler (see, e.g., \cite{BozFen1, BozFen2}) developed Herz's ideas, and Haagerup and his collaborators and students demonstrated the central role in harmonic analysis and operator theory of completely bounded multipliers.
Comparatively recently, G.~Pisier \cite{Pis1, Pis2} has studied uniformly bounded representations, on the one hand taking giant strides towards the solution of the Dixmier similarity problem and on the other developing the links between uniformly bounded representations and multipliers of the Fourier algebra.

Quite a lot of work has already been done about simple Lie groups of real rank one, that is members of the families $\group{SO}(n,1)$, $\group{SU}(n,1)$ and $\group{Sp}(n,1)$, as well as the single exceptional group $F_{4, -20}$.
For the moment, let us just recall that these groups contain a maximal compact subgroup $K$, and admit Iwasawa decompositions $KAN$, where $A$ is isomorphic to $\R$ and $N$ is nilpotent.
The \emph{spherical functions} $\phi_\zeta$ are $K$-bi-invariant functions on $G$ (with additional properties) that are important in harmonic analysis on $G$; in particular, they are matrix coefficients of the class-one principal series of representations.
J.~De Canni\`ere and Haagerup \cite{DcaHaa} and M.~Cowling and Haagerup \cite{CowHaa} estimated the completely bounded multiplier norms of various spherical functions.
T.~Steenstrup \cite{Stp} computed the norms $\lnorm \phi_\zeta \rnorm_{\cb}$ exactly for the bounded spherical functions on $\group{SO}(n,1)$, and showed that
\[
\begin{aligned}
\lnorm \phi_{\zeta} \rnorm_{\cb}
&= 	 \labs \frac{1 +  \sec^2(\pi \xi)  \sinh^2(\pi \eta)  } 	{  1 +  \sinh^2(\pi \eta)   } \rabs^{1/2}  ,
\end{aligned}
\]
where $\zeta = \xi + i \eta$.
For a fixed $\xi$, the right hand side is bounded and even, takes its minimum value $1$ when $\eta = 0$, and tends to $\sec(\pi \xi)$ as $\eta \to \pm \infty$.
Steenstrup deduced (following the ideas of Haagerup \cite{Haa2}) that there exist completely bounded multipliers of $A(G)$ that do not arise as a matrix coefficient of a uniformly bounded representation, using functional analysis and estimates for various norms associated to spherical functions.
This fact confirms that inequality~{\eqref{eq:cb-ub}} is far from sharp.

Our aim is to prove the Main Theorem; this paper is structured as follows.
In Section~2, we fix notation and review a few facts on uniformly bounded representations and completely bounded multipliers.
In Section~3, we describe the representations of the group $\group{SL}(2,\R)$, the intertwining operators, and the spherical functions.
In Section~4, we study the completely bounded multiplier norm of the generalised spherical functions, proving the Main Theorem.

\section{Background and notation}

In this section, we introduce some notation, and prove some preliminary results about completely bounded multipliers.
We also describe some formulae involving gamma functions that we will use.

Suppose that $G$ is a locally compact group.
We equip $G$ with right-invariant Haar measure (written $dx$ or $dy$ in integrals) and then define the usual Lebesgue spaces $L^p(G)$.
P.~Eymard \cite{Eym} defined the Fourier algebra $A(G)$, a space of continuous functions on $G$ that vanish at infinity.
The functions $u: G \to \C$ in $A(G)$ are the matrix coefficients of the right regular representation $\rho$ of $G$ on $L^2(G)$.
More precisely, $u \in A(G)$ if and only if there are functions $h$ and $k$ in $L^2(G)$ such that $u = \lip \rho h,k\rip$, that is,
\[
u(x) = \int_G h(yx) \,\bar k(y) \,dy
\qquad\forall x \in G.
\]
Eymard showed that $A(G)$ is closed under pointwise operations, which is not apparent from the definition, and further that $A(G)$ is a Banach algebra, with the norm given by
\[
\lnorm u \rnorm_A = \inf\lset \lnorm h \rnorm_2 \lnorm k \rnorm_2 : u = \lip \rho h,k\rip \rset.
\]
For future purposes, we note that if $G$ is a locally compact abelian group, written additively, with dual group $\hat G$, then $u$ is in $A(G)$ if and only if its Fourier transform $\hat u$ is in $L^1(\hat G)$; further, $\lnorm u \rnorm_{A} = \biglnorm \hat u \bigrnorm_{1}$.
Thus if
\[
u(x) = \int_G h(y + x) \, k(y) \, dy
\qquad\forall x \in G,
\]
where $h, k \in L^2(G)$, then
\[
\lnorm u \rnorm_A = \int_{\hat G} \labs \hat h(z) \, (\hat k)(-z) \rabs \, dz.
\]
If moreover $\biglabs \hat h\bigrabs = \biglabs (\hat k)\check{\phantom{k}}\bigrabs$, where $\check m (x) = m(-x)$ for all $x \in G$, then
\begin{equation}\label{eq:abelian-A-norm}
\lnorm u \rnorm_A = \lnorm h\rnorm_2 \lnorm k \rnorm_2.
\end{equation}

Eymard also defined the Fourier--Stieltjes algebra $B(G)$ of $G$, as the space of all matrix coefficients of ``all''  unitary representations of $G$.
Then $u \in B(G)$ if and only if there exist a unitary representation $\sigma$ of $G$ and vectors $\dvec{v}$ and $\dvec{w}$ in $\Hilb_\sigma$ such that $u = \lip\sigma \dvec{v}, \dvec{w} \rip$.
The $B(G)$-norm of $u$ is the infimum (in fact, the minimum) of the products $\lnorm \dvec{v} \rnorm_{\Hilb_\sigma} \lnorm \dvec{w} \rnorm_{\Hilb_\sigma}$ over all such representations of $u$ as $\sigma$, $\dvec{v}$ and $\dvec{w}$ all vary.

Note that inversion (the map $u \mapsto \check u$, where $\check u(x) = u(x^{-1})$) and complex conjugation of functions are isometries of $B(G)$.
Indeed, if $u = \lip\sigma \dvec{v}, \dvec{w} \rip$, then $\bar u \check{\phantom{x}} = \lip\sigma \dvec{w}, \dvec{v} \rip$, while $\bar u$ is a matrix coefficient of the contragredient representation.

We say that $v:G \to \C$ is a multiplier of the Fourier algebra, and we write $v \in MA(G)$, if the pointwise product $uv$ lies in $A(G)$ for all $u$ in $A(G)$; the multiplier norm of $v$, written $\lnorm v \rnorm_{MA}$, is the operator norm of the map $u \mapsto uv$.
Suppose that $\pi$ is a uniformly bounded representation of $G$ on a Hilbert space $\Hilb_\pi$ and $\dvec{v}, \dvec{w} \in \Hilb_\pi$, and suppose that $v = \lip\pi\dvec{v},\dvec{w}\rip$, that is, $v(x) = \lip\pi(x)\dvec{v},\dvec{w}\rip$ for all $x \in G$.
Then $v \in MA(G)$.
This may be proved by showing that $v$ satisfies one of the equivalent conditions for being a completely bounded multiplier of $A(G)$ in Lemma \ref{lem:equiv} below, which implies the desired result.

The dual of $A(G)$ may be identified with the von Neumann algebra $VN(G)$ of bounded convolution operators on $L^2(G)$ (acting on the right).
A completely bounded multiplier of $A(G)$ is a multiplier of $A(G)$ whose transpose is completely bounded as a map of $VN(G)$; this means that the transpose extends to a bounded map of $\mathcal{B}(L^2(G))$, or to a bounded map of $VN(G) \otimes \mathcal{B}(\Hilb)$ (where the tensor product is appropriately defined).
We refer to V.~Paulsen \cite{Pau} for much more about completely bounded maps; here we just recall a few well known results.
In what follows, we view the projective tensor product $L^2(G) \hatotimes L^2(G)$ as a subspace of $L^2(G \times G)$ in the natural way.

\begin{lemma} \label{lem:equiv}
Given a function $v$ on $G$, define the function $V: G \times G \to \C$ by
\begin{equation}\label{eq:extension-to-G^2}
V(x,y) = v(y^{-1}x)
\qquad\forall x, y \in G.
\end{equation}
Then the following conditions are equivalent:
\begin{enumerate}[(i)]
\item $v$ is a completely bounded multiplier of the Fourier algebra $A(G)$;
\item $V$ multiplies the projective tensor product $L^2(G) \hatotimes L^2(G)$ pointwise;
\item  $V(x, y) = \lip P(x), Q(y) \rip$ for all $x, y \in G$, where $P, Q: G \to \Hilb$ are continuous norm-bounded Hilbert-space-valued functions on $G$.
\end{enumerate}
The natural norms associated to each of these conditions are equal; these norms are the completely bounded multiplier norm, the multiplier norm and the infimum of the products $\lnorm P\rnorm_{\infty} \lnorm Q \rnorm_{\infty}$,  over all representations of $V$ as in item (iii).
\end{lemma}

We write $\lnorm v \rnorm_{\cb}$ for the norm in item (i).
As the terminology suggests, completely bounded multipliers are multipliers, and moreover
\[
\lnorm v \rnorm_{MA} \leq \lnorm v\rnorm_{\cb}.
\]
We remark that sometimes completely bounded multipliers are known as Herz--Schur multipliers, though this latter term may also apply to multipliers of the tensor product.

In any case, if $v = \lip\pi\dvec{v},\dvec{w}\rip$, where $\pi$ is a uniformly bounded representation of $G$ and $\dvec{v},\dvec{w}\in \Hilb_\pi$ are unit vectors, then $V(x,y) = \lip \pi(x)\dvec{v}, \pi(y^{-1})^*\dvec{w} \rip$ for all $x, y  \in G$, and so
\begin{equation*}
\lnorm v \rnorm_{\cb} \leq \lnorm \pi \rnorm_{\ub}^2.
\end{equation*}
The question now arises as to whether this inequality is sharp.
A related question is whether the inequality $\lnorm v \rnorm_{MA} \leq \lnorm v\rnorm_{\cb}$ is sharp.
Similar questions are treated by \cite{Boz} and \cite{Los}.

\medskip

The main ingredient in the proof of our estimates is the following variant of a result in~\cite[p.~516]{CowHaa}.

\begin{proposition}
\label{uno-gen}
Let $G$ be a locally compact group and $K$ be a compact subgroup of $G$.
Suppose that $S$ is an amenable closed subgroup of $G$ such that $G=SK$
and $S\cap K$ is trivial.
Let $\chi_1$ and $\chi_2$ be unitary characters of $K$ and suppose that $\phi$ is a function on $G$ such that
\begin{equation}\label{eq:K-bi-equivariance}
\phi(k_1 x k_2) = \chi_1(k_1)\,\chi_2(k_2)\,\phi(x)
\qquad\forall x\in G \quad\forall k_1,k_2\in K.
\end{equation}
Then $\phi$ is a completely bounded multiplier of $A(G)$ if and only if $\phi{|_{S}}$ is in $B(S)$; moreover,
\[
\lnorm \phi \rnorm_{\cb}= \lnorm \phi |_{S} \rnorm_{B(S)} \, .
\]
\end{proposition}

\begin{proof}
Suppose that $\phi|_{S}$ is in  $B(S)$. Then
there exist a unitary representation $\sigma$
of $S$ and two vectors $\dvec{v}$ and $\dvec{w}$ in  $\Hilb_\sigma$ such that
\[
\phi(s)=\langle \sigma(s)\dvec{v},\dvec{w}\rangle_{\Hilb_\sigma}
\qquad \forall s\in S.
\]

We define $P:G\to \Hilb_\sigma$ and $Q:G\to \Hilb_\sigma$ by
\[
\begin{aligned}
P(sk)&=\sigma(s)(\chi_2(k)\dvec{v})=\chi_2(k)\sigma(s)\dvec{v}
\\
Q(sk)&=\sigma(s)(\chi_1(k)\dvec{w})=\chi_1(k)\sigma(s)\dvec{w}
\end{aligned}
\]
for all $s \in S$ and $k \in K$; as $G = SK$ and $S\cap K$ is trivial, $P$ and $Q$ are well defined.

Note that if $x_1=s_1k_{1}$ and $x_2=s_2k_{2}$ are in $G$ then
\begin{align*}
\phi(x_1^{-1}x_2)
&=
\phi(k_{1}^{-1}s_1^{-1}s_2k_{2})
=
\chi_1(k_1^{-1})\, \chi_2(k_2)\,\phi(s_1^{-1}s_2)
\\
&=\overline{\chi_1(k_1)}\, \chi_2(k_2)\,\langle \sigma(s_1^{-1}s_2)\dvec{v},\dvec{w}\rangle_{\Hilb_\sigma}
\\
&=\overline{\chi_1(k_1)}\, \chi_2(k_2)\,\langle \sigma(s_2)\dvec{v},\sigma(s_1)\dvec{w}\rangle_{\Hilb_\sigma}
\\
&=\langle P(x_2),Q(x_1)\rangle_{\Hilb_\sigma}.
\end{align*}
Therefore
\begin{align*}
\lnorm \phi \rnorm_{\cb}
&\leq \sup\lset\lnorm P(x) \rnorm_{\Hilb_\sigma}
\lnorm Q(y) \rnorm_{\Hilb_\sigma}\,:\, x,y\in G\rset
\\
&=\sup\lset
\lnorm \sigma(s_2)(\chi_2(k_2)\dvec{v}) \rnorm_{\Hilb_\sigma}\,
\lnorm \sigma(s_1)(\chi_1(k_1)\dvec{w}) \rnorm_{\Hilb_\sigma}
\,:\,s_1,s_2\in S,\, k_1,k_2\in K \rset
\\
&=
\lnorm \dvec{v} \rnorm_{\Hilb_\sigma}\,
\lnorm \dvec{w} \rnorm_{\Hilb_\sigma} .
\end{align*}
It follows that
\[
\begin{aligned}
\lnorm \phi \rnorm_{\cb}&\leq \inf\lset \lnorm \dvec{v} \rnorm_{\Hilb_\sigma}\,
\lnorm \dvec{w} \rnorm_{\Hilb_\sigma} : \phi=\lip\sigma\dvec{v},\dvec{w}\rip_{\Hilb_\sigma} \,\,
\sigma \text{ unitary representation of }S \rset
\\
&=
\lnorm \phi|_{S} \rnorm_{B(S)}.
\end{aligned}
\]
The reverse inequality may be proved as in \cite{CowHaa}.
\end{proof}

\begin{corollary}\label{cor:ub-norm-identities}
Let $G$ be a locally compact group and $\phi$ be a function on $G$ which transforms as in Proposition \ref{uno-gen}.
Then
\[
\lnorm \bar\phi \rnorm_{\cb} = \lnorm \check \phi \rnorm_{\cb} = \lnorm \phi \rnorm_{\cb}.
\]
\end{corollary}

\begin{proof}
This follows immediately from the earlier observation that
\[
\lnorm \bar\phi |_S \rnorm_{B}
= \lnorm \check \phi |_S \rnorm_{B}
= \lnorm \phi |_S \rnorm_{B}
\]
and Proposition \ref{uno-gen}.
\end{proof}

We are going to use the methods of Fourier analysis, and define the Fourier transform $\cF f$ or $\hat f$  of a function $f$ on $\R$ by
\[
\cF F(y) = \hat f(y) = \int_\R f(x)\,e^{-ixy}\,dx
\qquad\forall y \in \R.
\]
Similarly, for a function $g$ on $[-\pi, \pi]$,  we define the Fourier coefficients $\tilde g(\kappa )$ by
\[
\tilde g(\kappa ) = \frac{1}{2\pi} \int_{-\pi}^{\pi} g(\theta) \,e^{-i\kappa  \theta} \, d\theta
\qquad\forall \kappa \in \Z.
\]

We recall here some gamma function formulae that we shall need later; apart from the \emph{recurrence relation} $\Gamma(z+1)=z\,\Gamma(z)$, we will use the \emph{reflection formula} (see~\cite[p.~3, formula (6)]{Erd})
\begin{equation}\label{eq:reflection}
\Gamma(z)\,\Gamma(1-z)=\pi\,\cosec(\pi z),
\end{equation}
the \emph{duplication formula} (see~\cite[p.~5, formula (15)]{Erd})
\begin{equation}\label{eq:duplication}
 \Gamma(z+\half) \, \Gamma(z)=2^{1-2z}\,\pi^{1/2}\,\Gamma(2z),
\end{equation}
\emph{Stirling's formula} (see~\cite[p.~47, formula (2)]{Erd})
\begin{equation}\label{eq:Stirling}
\Gamma(z)
=(2\pi)^{1/2} \,e^{-z}\, e^{(z-\half)\log z} \left( 1+O(z^{-1})\right)
\qquad \text{as  $z\to \infty$ and $\labs \arg z \rabs \leq \dvec{w}$} ,
\end{equation}
where $0 \leq \dvec{w} < \pi$, and \emph{Dougall's formula} (see \cite[p.~7, formula (1)]{Erd})
\begin{equation}\label{eq:Dougall}
\sum_{k \in \Z}
\frac{\Gamma(a+k)\, \Gamma(b+k)}{\Gamma(c+k)\, \Gamma(d+k)}
=\pi^2 \cosec(\pi a) \cosec(\pi b)\,
\frac{\Gamma(c+d-a-b-1)}{\Gamma(c-a)\, \Gamma(d-a)\, \Gamma(c-b)\, \Gamma(d-b)} \,,
\end{equation}
where $a,b \in \C \setminus \Z$, $c,d \in \C$ and $\Re{a+b-c-d}<-1$.

Note also that $\Gamma(\bar z) = \bar\Gamma(z)$, so $\labs \Gamma(\bar z) \rabs = \labs \Gamma(z) \rabs$.

We will often see certain ratios of gamma functions, and note that if $k \in \Z$ and $z \in \C$, and all terms are well-defined, then
\begin{equation}\label{eq:gamma-ratios}
\frac{\Gamma(\half + z + \frac{k}{2})}{\Gamma(\half - z + \frac{k}{2})}
= (-1)^k \frac{\Gamma(\half + z - \frac{k}{2})}{\Gamma(\half - z - \frac{k}{2})} \,.
\end{equation}
To see this, suppose that $k \geq 0$ and use the recurrence relation expand out the ratios
\[
\frac{\Gamma(\half + z + \frac{k}{2})}{\Gamma(\half + z - \frac{k}{2})}
\qquad\text{and}\qquad
\frac{\Gamma(\half - z + \frac{k}{2})}{\Gamma(\half - z - \frac{k}{2})} \,,
\]
and note that the factors are equal up to signs.
Note also that
\begin{equation}\label{eq:gamma-ratios-limit}
\labs \frac{\Gamma(\half + z + \frac{k}{2})} {\Gamma(\half - z + \frac{k}{2})} \rabs
\simeq
\labs k\rabs ^{2\Re z}
\end{equation}
as $k \to +\infty$ from Stirling's formula;  in light of \eqref{eq:gamma-ratios}, this also holds as $k \to -\infty)$.

We also use an integral formula (see \cite[p.~12, formula (30)]{Erd}):
\begin{equation}\label{eq:integral-formula}
\int_0^{\pi/2} \cos^{\alpha}(\theta) \cos(\beta\theta) \,d\theta
= 2^{-1-\alpha} \,\pi \, \frac{\Gamma(1 + \alpha)}{\Gamma(1+\frac{\alpha}{2} + \frac{\beta}{2}) \, \Gamma(1+\frac{\alpha}{2} - \frac{\beta}{2})} \,,
\end{equation}
when $\Re\alpha > -1$ and $\beta  \in \Z$, which gives us the following lemma.

\begin{lemma}\label{lem:definite-integral}
Suppose that $\epsilon \in \{0,1\}$, $\Re\alpha > -1$, and $\kappa  \in \Z$.
Then
\[
\int_{-\pi}^{\pi} e^{i\kappa  \theta} \sgn^\epsilon(\sin(\theta)) \labs\sin(\theta)\rabs^{\alpha} \,d\theta
= \sigma(\kappa ,\epsilon) \, 2^{1 - \alpha} \,\pi \, \frac{\Gamma(1 + \alpha)}{\Gamma(1+\frac{\alpha}{2} + \frac{\kappa }{2}) \, \Gamma(1+\frac{\alpha}{2} - \frac{\kappa }{2})} \,,
\]
where $\sigma(\kappa ,\epsilon) = (1-\epsilon) \cos(\kappa \pi/2) + i\epsilon \sin(\kappa \pi/2)$.
\end{lemma}

\begin{proof}
If $\epsilon = 0$, then by parity, a change of variable, and parity,
\begin{align*}
&\int_{-\pi}^{\pi} e^{i\kappa \theta} \sgn^\epsilon(\sin(\theta)) \labs\sin(\theta)\rabs^{\alpha} \,d\theta
\\
&\quad= 2\int_{0}^{\pi} \cos(\kappa \theta) \labs\sin(\theta)\rabs^{\alpha} \,d\theta
\\
&\quad= 2\int_{-\pi/2}^{\pi/2} \cos(\kappa (\theta+\pi/2)) \labs\cos(\theta)\rabs^{\alpha} \,d\theta
\\
&\quad= 2\cos(\kappa \pi/2) \int_{-\pi/2}^{\pi/2} \cos(\kappa \theta) \labs\cos(\theta)\rabs^{\alpha} \,d\theta
\\
%&\quad= 4\cos(\kappa \pi/2) \int_{0}^{\pi/2} \cos(\kappa \theta) \labs\cos(\theta)\rabs^{\alpha} \,d\theta \\
&\quad=  \cos(\kappa \pi/2) \, 2^{1 - \alpha} \, \pi \,
\frac{\Gamma(1 + \alpha)}{\Gamma(1+\frac{\alpha}{2} + \frac{\kappa }{2}) \, \Gamma(1+\frac{\alpha}{2} - \frac{\kappa }{2})} \,,
\end{align*}
by parity and \eqref{eq:integral-formula}.
This expression is $0$ if $\kappa $ is odd.

If $\epsilon = 1$, then similarly
\begin{align*}
&\int_{-\pi}^{\pi} e^{i\kappa \theta} \sgn^\epsilon(\sin(\theta)) \labs\sin(\theta)\rabs^{\alpha} \,d\theta
\\
&\quad= 2i\int_{0}^{\pi} \sin(\kappa \theta) \labs\sin(\theta)\rabs^{\alpha} \,d\theta
\\
&\quad= 2i\int_{-\pi/2}^{\pi/2} \sin(\kappa (\theta+\pi/2)) \labs\cos(\theta)\rabs^{\alpha} \,d\theta
\\
&\quad= 2i\sin(\kappa \pi/2) \int_{-\pi/2}^{\pi/2} \cos(\kappa \theta) \labs\cos(\theta)\rabs^{\alpha} \,d\theta
\\
%&\quad= 4i\sin(\kappa \pi/2) \int_{0}^{\pi/2} \cos(\kappa \theta) \labs\cos(\theta)\rabs^{\alpha} \,d\theta \\
&\quad=  i \sin(\kappa \pi/2) \, 2^{1 - \alpha} \, \pi \,
\frac{\Gamma(1 + \alpha)}{\Gamma(1+\frac{\alpha}{2} + \frac{\kappa }{2}) \, \Gamma(1+\frac{\alpha}{2} - \frac{\kappa }{2})}  \,,
\end{align*}
by parity and \eqref{eq:integral-formula}.
This expression is $0$ if $\kappa $ is even.
\end{proof}

\section{The group $\group{SL}(2,\R)$}\label{sec:sl2r}

In this section, we first describe the group $\group{SL}(2, \R)$, then the principal series of representations of the group; third, we analyse the intertwining operators for these representations, and finally we define and discuss the generalised spherical functions.

We now describe $\group{SL}(2, \R)$, abbreviated to $G$ for convenience, and various decompositions and representations thereof.
We present an approach that the second-named author learnt from R.A.~Kunze many years ago.
First, define subgroups $K$, $M$, $A$, $N$ and $\bar N$ of $G$ as follows:
\[
\begin{gathered}
K = \lset k_\theta : \theta \in \R \rset \qquad M = \lset m_{\pm}  \rset \qquad A = \lset a_s : s \in \R^+ \rset \\ N = \lset n_t : t \in \R \rset \qquad \bar N = \lset \bar n_t : t \in \R \rset,
\end{gathered}
\]
where
\[
\begin{gathered}
k_\theta = \begin{pmatrix} \cos\theta & \sin\theta \\ -\sin\theta & \cos\theta \end{pmatrix} = \exp\lpar \theta \begin{pmatrix} 0 & 1 \\ -1 & 0 \end{pmatrix} \rpar
\\
m_\pm = \begin{pmatrix} \pm 1 &  0 \\  0 & \pm 1 \end{pmatrix}
\qquad
a_s = \begin{pmatrix} s &0 \\ 0  & s^{-1} \end{pmatrix} =  \exp\lpar \log(s)  \begin{pmatrix} 1 & 0\\ 0 & -1 \end{pmatrix} \rpar
\\
n_t = \begin{pmatrix} 1 & 0 \\ t & 1 \end{pmatrix} = \exp\lpar t \begin{pmatrix} 0 & 0\\ 1 & 0 \end{pmatrix} \rpar
\qquad
\bar n_t = \begin{pmatrix} 1 & t \\ 0 & 1 \end{pmatrix} = \exp\lpar t \begin{pmatrix} 0 & 1\\ 0 & 0 \end{pmatrix} \rpar;
\end{gathered}
\]
we will write $w$ for the rotation $k_{\pi/2}$.

There are a number of standard decompositions of $G$.
The Iwasawa decomposition asserts that every element $x$ of $G$ may be expressed uniquely in the form $x = \bar nak$, where $\bar n \in \bar N$, $a \in A$ and $k \in K$ (this may also be stated with $N$ in place of $\bar N$).

The Bruhat decomposition affirms first that $NAM$ is a subgroup of $G$ and each element $x$ of $NAM$ has a unique expression in the form $x = \bar nam$, where $\bar n \in \bar N$, $a \in A$ and $m \in M$; and second that $G$ is the disjoint union $(NAM) \sqcup (NAM w NAM)$.

The third frequently used decomposition, the Cartan decomposition, states that every element of $G$ may be written in the form $k_\theta a_s k_\phi$, where $s \geq 1$; if $x \in K$ then $s = 1$ and there are many choices for $k_\theta$ and $k_\phi$; otherwise, the decomposition is unique up to changes in $\theta$ and $\phi$ by adding (or subtracting) $\pi$ to both or $2\pi$ to either.
This result is derived using the polar decomposition of a matrix.

Consider $\R^2$ as a space of row vectors, and let $G$ act on $\R^2$ by right multiplication.
Then $G$ fixes the origin and acts transitively on $\R^2 \setminus \{(0,0)\}$.

We now consider the space $\vect_{\zeta,\epsilon}$, where
$\zeta \in \C$ and $\epsilon \in \{0,1 \}$, of smooth functions on $\R^2 \setminus\{(0,0)\}$ that satisfy the homogeneity and parity conditions
\begin{align}
f(\delta v) &= \delta^{2\zeta-1} f(v)
\qquad\forall v \in \R^2 \setminus\{(0,0)\} \quad\forall \delta \in \R*+ ,\label{eq:homogeneity} \\
f(- v) &= (-1)^\epsilon f(v)
\qquad\forall v \in \R^2 \setminus\{(0,0)\}  .\label{eq:parity}
\end{align}
We equip $\vect_{\zeta,\epsilon}$ with the topology of locally uniform convergence of all partial derivatives.
The partial derivatives may be taken in polar coordinates; for homogeneous functions, the radial derivatives are easy to deal with, and questions of convergence boil down to the behaviour of the angular derivatives on the unit circle.

Since $G$ acts on $\R^2 \setminus\{(0,0)\}$ and commutes with scalar multiplication, $G$ acts on $\vect_{\zeta,\epsilon}$ by the formula
\[
\pi_{\zeta,\epsilon}(x) f(v) = f(vx)
\qquad\forall v \in \R^2 \setminus\{(0,0)\}\quad\forall x \in G.
\]
We obtain the ``compact picture'' of the representation by restricting $v$ to lie on the unit circle $\{ (s,t) \in \R^2 : s^2 + t^2 = 1 \}$, and observing that
\[
\pi_{\zeta,\epsilon}(x) f(v) = \labs vx \rabs^{2\zeta-1} f( \labs vx \rabs^{-1} vx) .
\]
Similarly, we obtain the ``noncompact picture'' by restricting $v$ to lie on the vertical line $\{ (1,t) : t \in \R \}$, and observing that
\[
\pi_{\zeta,\epsilon}(x) f(1,t) = \sgn^\epsilon (a+tc) \labs a+tc \rabs^{2\zeta-1} f( 1, x \cdot t) ,
\]
where
\[
x = \begin{pmatrix} a & b \\ c& d \end{pmatrix}
\qquad\text{and}\qquad
x \cdot t = \dfrac{b+dt}{a+ct}.
\]
Later we will consider this representation acting on functions on $\R$, and we will write
\begin{equation}\label{eq:def-pi-R}
\pi^\R_{\zeta,\epsilon}(x) f(t) = \sgn^\epsilon (a+tc) \labs a+tc \rabs^{2\zeta-1} f(x \cdot t) ,
\end{equation}
where $f$ is a function on $\R$.
Clearly some care is required ``at infinity'' in this version of the representation.
Note that $\pi^\R_{\zeta,0}$ and $\pi^\R_{\zeta,1}$ coincide on $\bar NA$, as when $x \in \bar NA$, $a+tc = a > 0$.

The space $\vect_{\zeta,\epsilon}$ is spanned (topologically) by functions whose restrictions to the unit circle are complex exponentials.
We write $f_{\zeta,\mu}$ for the function in $\vect_{\zeta,\epsilon}$ determined by the condition that
\begin{equation}\label{eq:def-vettori}
f_{\zeta,\mu} (\cos\theta,\sin\theta) = \frac{1}{\pi^{1/2}}\,e^{i\mu\theta}
\qquad\forall \theta \in \R,
\end{equation}
with the restriction that $\mu - \epsilon$ must be even; every function in $\vect_{\zeta,\epsilon}$ is the limit of finite linear combinations of the functions $f_{\zeta, \mu}$.
It is easy to check that
\begin{equation}\label{eq:basis-on-R}
f_{\zeta,\mu} (1, t)
%= \frac{1}{\pi^{1/2}} (1+t^2)^{\zeta - \half - \frac{\mu}{2}}  (1 + it)^\mu
= \frac{1}{\pi^{1/2}} (1+it)^{\zeta - \half + \frac{\mu}{2}}  (1 - it)^{\zeta - \half - \frac{\mu}{2}},
\end{equation}
and so in particular, $f_{\zeta,-\mu} (1,t) = f_{\zeta,\mu} (1,-t)$ for all $t \in \R$.

\begin{lemma}\label{lem:basis}
Every function in $\vect_{\zeta,\epsilon}$ is a sum $\sum_{\mu \in 2\Z + \epsilon} a_\mu \, f_{\zeta,\mu}$, where the coefficients $a_\mu$ tend to zero faster than any polynomial in $\mu$.
Conversely, every such sum is a function in $\vect_{\zeta,\epsilon}$.
\end{lemma}

\begin{proof}
Every function in $\vect_{\zeta,\epsilon}$ is determined by its restriction to the unit circle, and the Fourier series of smooth functions on the unit circle are exactly of this form.
\end{proof}

\begin{lemma}\label{lem:pairing}
Suppose that $f \in \vect_{\zeta,\epsilon}$ and $g \in \vect_{-\zeta,\epsilon}$.
Then
\begin{align*}
\int_\R  f(1,t) \, g(1,t) \,dt
&= \int_{-\pi/2}^{\pi/2}  f(\cos\theta,\sin\theta ) \, g(\cos\theta,\sin\theta ) \,d\theta \\
&= \frac{1}{2} \int_{-\pi}^{\pi}  f(\cos\theta,\sin\theta ) \, g(\cos\theta,\sin\theta ) \,d\theta.
\end{align*}
We define the pairing $(f, g)$ to be any of the above integrals, then
\begin{equation}\label{eq:invariant-pairing}
(\pi_{\zeta,\epsilon} (x) f, \pi_{-\zeta,\epsilon} (x) g) = (f, g)
\qquad\forall x \in G,
\end{equation}
or equivalently,
\[
\pi_{-\zeta,\epsilon} (x)^\top = \pi_{\zeta,\epsilon} (x^{-1})
\qquad\forall x \in G,
\]
where $\pi_{-\zeta,\epsilon} (x)^\top$ denotes the transpose of $\pi_{-\zeta,\epsilon} (x)$.
\end{lemma}

\begin{proof}
The first integral is equal to the second by the change of variables $t = \tan(\theta)$, and the second is equal to the third because $fg$ is even.

It is obvious that \eqref{eq:invariant-pairing} holds when $x \in \bar NA$ from a simple change of variable in the integral over $\R$ and when $x \in K$ from a change of variable in the integral over $[-\pi,\pi]$.
By the Iwasawa decomposition, $\bar NA$ and $K$ generate $G$, and so \eqref{eq:invariant-pairing} holds for all $x \in G$.
This implies that
\[
(\pi_{-\zeta,\epsilon} (x)^\top \pi_{\zeta,\epsilon} (x) f,  g) = (f, g)
\qquad\forall x \in G.
\]
The bilinear form $(\cdot,\cdot)$ gives us a duality between $\vect_{\zeta,\epsilon}$ and $\vect_{-\zeta,\epsilon}$.
While $\vect_{-\zeta,\epsilon}$ is smaller than the topological dual space of $\vect_{\zeta,\epsilon}$, it is weak-star dense in the dual space, and the set $\{ (f,g) : g \in \vect_{-\zeta, \epsilon} \}$ determines $f$ in $\vect_{\zeta,\epsilon}$.
Thus $\pi_{-\zeta,\epsilon} (x)^\top \pi_{\zeta,\epsilon} (x)$ is the identity operator on $\vect_{\zeta,\epsilon}$, and hence $\pi_{-\zeta,\epsilon} (x)^\top = \pi_{\zeta,\epsilon} (x^{-1})$.
\end{proof}

Note that our normalisation of $f_{\zeta,\mu}$ means that $( f_{\zeta,\mu}, f_{-\zeta,\nu} )$ is equal to $1$ if $\mu + \nu = 0$ and is equal to $0$ otherwise.

\begin{corollary}
Suppose that $\Re\zeta = 0$.
The representation $\pi_{\zeta,\epsilon}$ acts unitarily when $\vect_{\zeta,\epsilon}$ is equipped with the inner product $\lip f, g \rip = ( f, \bar g)$ and completed to a Hilbert space.
\end{corollary}

\begin{proof}
Evidently, $\lip \cdot, \cdot \rip$ is an inner product on $\vect_{\zeta,\epsilon}$.
Further, $(\pi_{\zeta, \epsilon} g) \bar{\phantom{x}} = \pi_{-\zeta, \epsilon} \bar g$, and so
\[
\lip \pi_{\zeta,\epsilon}(x) f, \pi_{\zeta,\epsilon}(x) g \rip
= ( \pi_{\zeta,\epsilon}(x) f, \pi_{-\zeta,\epsilon}(x) \bar g)
= (f, \bar g)
= \lip f, g \rip,
\]
and each $\pi_{\zeta,\epsilon}(x)$ is unitary.
\end{proof}

This Hilbert space may be identified with the standard $L^2$ space on the unit circle or on the line $\{ (1,t) : t \in \R \}$.
We may show similarly that the representations $\pi_{\zeta,\epsilon}$ act isometrically on $L^p$ spaces when $p(\Re\zeta - \half) = 1$; see \cite{ACD} for the details.

\begin{lemma}
Suppose that $f$ and $g$ are continuous on $\R^2 \setminus\{(0,0)\}$ and satisfy the homogeneity condition \eqref{eq:homogeneity}, where $- \half < \Re\zeta < 0$.
Then
\begin{gather}
\int_\R\int_\R f(1,t) \, g(1,u) \sgn^\epsilon(u - t) \labs u-t\rabs^{-1-2\zeta} \,dt\,du \\
\notag
= \int_{-\pi/2}^{\pi/2} \int_{-\pi/2}^{\pi/2}  f(\cos\theta, \sin\theta) \, g(\cos\phi, \sin\phi) \sgn^\epsilon(\sin(\phi - \theta)) \labs \sin(\phi-\theta)\rabs^{-1-2\zeta} \,d\phi\,d\theta.
\end{gather}
If moreover $f$ and $g$ satisfy the parity condition \eqref{eq:parity}, then these double integrals are both equal to
\[
\frac{1}{4} \int_{-\pi}^{\pi} \int_{-\pi}^{\pi}  f(\cos\theta, \sin\theta) \, g(\cos\phi, \sin\phi) \sgn^\epsilon(\sin(\phi - \theta)) \labs \sin(\phi-\theta)\rabs^{-1-2\zeta} \,d\phi\,d\theta.
\]
For $f, g \in \vect_{\zeta,\epsilon}$, write $(I_{\zeta,\epsilon} f, g)$ for any of the three double integrals above.
Then
\begin{equation}\label{eq:double-integral-invariance}
(I_{\zeta,\epsilon} \pi_{\zeta,\epsilon}(x) f, \pi_{\zeta,\epsilon}(x) g) = (I_{\zeta,\epsilon} f, g)
\qquad\forall x \in G.
\end{equation}
\end{lemma}

\begin{proof}
Since
\[
( \tan\theta - \tan\phi ) \cos\theta \cos\phi
= \sin(\theta - \phi),
\]
we see that $\sgn(\tan\phi - \tan\theta) = \sgn(\sin(\phi-\theta))$ when $\theta, \phi \in (-\pi/2, \pi/2)$, and hence
\[
\begin{aligned}{}
&\int_{\R} \int_{\R} f(1,t) \, g(1,u) \sgn^\epsilon(u-t) \labs u - t\rabs ^{-1-2\zeta} \,dt \,du \\
&\quad = \int_{-\pi/2}^{\pi/2}  \int_{-\pi/2}^{\pi/2} f(1,\tan\theta) \, g(1, \tan\phi) \\
&\qquad\qquad\qquad\qquad \sgn^\epsilon(\tan\phi - \tan\theta) \labs \tan\phi - \tan\theta \rabs ^{-1-2\zeta} \sec^2\theta \sec^2\phi \,d\theta\,d\phi \\
&\quad =  \int_{-\pi/2}^{\pi/2}  \int_{-\pi/2}^{\pi/2} f(\cos\theta,\sin\theta ) \,g(\cos\phi,\sin\phi )  \\
&\qquad\qquad\qquad\qquad \sgn^\epsilon(\tan\phi - \tan\theta) \labs \tan\phi - \tan\theta \rabs ^{-1-2\zeta} \sec^{1+2\zeta} \theta \sec^{1+2\zeta} \phi \,d\theta\,d\phi \\
&\quad =  \int_{-\pi/2}^{\pi/2}  \int_{-\pi/2}^{\pi/2} f(\cos\theta,\sin\theta ) \, g(\cos\phi,\sin\phi ) \\
&\qquad\qquad\qquad\qquad \sgn^\epsilon(\sin(\phi-\theta)) \labs \sin(\theta-\phi) \rabs ^{-1-2\zeta} \,d\theta\,d\phi ;
\end{aligned}
\]
all integrals converge absolutely.
If moreover $f$ and $g$ both satisfy the parity condition, the last integral is equal to
\[
\begin{aligned}{}
&\frac{1}{4}\int_{-\pi}^{\pi}  \int_{-\pi}^{\pi} f(\cos\theta,\sin\theta ) \, g(\cos\phi,\sin\phi ) \\
&\qquad\qquad\qquad\qquad \sgn^\epsilon(\sin(\phi-\theta)) \labs \sin(\theta-\phi) \rabs ^{-1-2\zeta} \,d\theta\,d\phi
\end{aligned}
\]
for reasons of parity.
We write any of the three double integrals above as $(I_{\zeta, \epsilon} f,g)$; at this stage, $I_{\zeta, \epsilon} f$ is given by various single integrals and may be a measurable function or a distribution.

Now, as in the definition of the pairing, we may show that for $f, g \in \vect_{\zeta,\epsilon}$,
\[
(I_{\zeta, \epsilon} \pi_{\zeta, \epsilon}(x) f, \pi_{\zeta, \epsilon}(x) g) = (I_{\zeta, \epsilon} f,g)
\qquad\forall x \in G,
\]
by considering the last integral when $x \in K$ and the first when $x \in \bar NA$, and using the Iwasawa decomposition.
We are going to prove a similar result later for $x \in \bar NA$, and omit the details here.
\end{proof}

\begin{lemma}\label{lem:intertwining-basis}
Suppose that $- \half < \Re\zeta < 0$, $\epsilon \in \{0,1\}$, and $\mu \in \Z$.
Then
\[
I_{\zeta,\epsilon} f_{\zeta, \mu} = b(\zeta,\mu) f_{-\zeta,\mu}
\quad\text{where}\quad
b(\zeta,\mu)
= \sigma(\mu,\epsilon) \, 2^{1+2\zeta} \, \pi \, \frac{\Gamma(-2\zeta)}{\Gamma(\half - \zeta + \frac{\mu}{2}) \, \Gamma(\half - \zeta - \frac{\mu}{2})} \,,
\]
where $\sigma(\mu,\epsilon) = (1-\epsilon) \cos(\mu\pi/2) + i\epsilon \sin(\mu\pi/2)$.
Further, $I_{\zeta,\epsilon}$ maps $\vect_{\zeta,\epsilon}$ bijectively and bicontinuously onto $\vect_{-\zeta,\epsilon}$, and
\begin{equation}\label{eq:intertwining-repns}
I_{\zeta,\epsilon} \pi_{\zeta,\epsilon}(x)
= \pi_{-\zeta,\epsilon}  I_{\zeta,\epsilon}(x)
\qquad\forall x \in G.
\end{equation}
\end{lemma}

\begin{proof}
We consider the formulae of the previous lemma in more depth.
We observe that
\[
(I_{\zeta,\epsilon} f_{\zeta,\mu}, f_{\zeta,\nu})
= (I_{\zeta,\epsilon} \pi_{\zeta,\epsilon}(k_\theta) f_{\zeta,\mu}, \pi_{\zeta,\epsilon}(k_\theta) f_{\zeta,\nu})
= e^{i(\mu+\nu)\theta} (I_{\zeta,\epsilon} f_{\zeta,\mu}, f_{\zeta,\nu}) ,
\]
and hence this is $0$ unless $\mu + \nu = 0$.
Thus $I_{\zeta, \epsilon} f_{\zeta, \mu}$, which \emph{a fortiori} is an element of the dual space of $\vect_{\zeta, \epsilon}$ and so a distribution, is actually a multiple of $f_{-\zeta, \mu}$ and hence a smooth function.
The multiple $b(\zeta,\mu)$ is equal to $(I_{\zeta,\epsilon} f_{\zeta,\mu}, f_{\zeta,-\mu})$, that is,
\begin{align*}
& \frac{1}{4\pi} \int_{-\pi}^\pi \int_{-\pi}^\pi e^{i\mu\theta} e^{-i\mu\phi} \sgn^\epsilon(\sin(\phi-\theta)) \labs \sin(\phi-\theta) \rabs^{-1-2\zeta} \,d\theta\,d\phi\\
& = \frac{1}{2} \int_{-\pi}^\pi e^{-i\mu\psi} \sgn^\epsilon(\sin\psi) \labs \sin(\psi) \rabs^{-1-2\zeta} \,d\psi\\
&= \sigma(\mu,\epsilon) \, 2^{1 + 2\zeta} \pi \, \frac{\Gamma(-2\zeta)}{\Gamma(\half - \zeta + \frac{\mu}{2}) \, \Gamma(\half - \zeta - \frac{\mu}{2})} \,,
\end{align*}
from Lemma \ref{lem:definite-integral}.

For a fixed $\zeta$ such that $-\half < \Re\zeta < 0$, the constant $b_{\zeta,\mu}$ does not vanish, and for a fixed $\zeta$, both the constant and its inverse grow at most polynomially in $\mu$ as $\mu$ tends to infinity.
In light of Lemma \ref{lem:basis}, the linear map $I_{\zeta,\epsilon}$ takes $\vect_{\zeta,\epsilon}$ bijectively and bicontinuously onto $\vect_{-\zeta,\epsilon}$.
Now \eqref{eq:intertwining-repns} follows from \eqref{eq:double-integral-invariance} and a duality argument, as in the proof of Lemma \ref{lem:pairing}.
\end{proof}

We now normalise the intertwining operators, following Kunze and Stein.
From \cite[p.~173]{GelShi} or \cite[p.~160]{StnWei}, we know that
\[
\cF (\sgn^{\epsilon}(\cdot) \labs\cdot\rabs^{-1-2\zeta})
= c(\zeta,\epsilon) \sgn^{\epsilon}(\cdot) \labs\cdot\rabs^{2\zeta},
\]
where
\begin{equation}\label{eq:def-c-lambda-epsilon}
c(\zeta,\epsilon) = i^{\epsilon} \pi^{1/2} 2^{-2\zeta} \frac{\Gamma(\frac{\epsilon}{2} - \zeta)}{\Gamma(\half + \frac{\epsilon}{2} +\zeta)} \,;
\end{equation}
we define
\begin{equation*}%\label{eq:def-normalised-intertwining}
J_{\zeta, \epsilon} = \frac{1}{c(\zeta,\epsilon)} \, I_{\zeta, \epsilon}.
\end{equation*}
\begin{lemma}
Suppose that $-\half < \Re\zeta < 0$ and $\epsilon \in \{0,1\}$.
Then the following hold:
\begin{enumerate}[(i)]
\item
$J_{\zeta,\epsilon}$ maps $\vect_{\zeta,\epsilon}$ bijectively and bicontinuously onto $\vect_{-\zeta,\epsilon}$;
\item
$J_{\zeta,\epsilon} f_{\zeta, \mu} = d(\zeta,\epsilon,\mu) f_{-\zeta,\mu}$, where
\[
d(\zeta,\epsilon,\mu) = 2^{2\zeta}
\, \frac{ \Gamma(\half + \zeta + \frac{\mu}{2}) } {\Gamma(\half - \zeta + \frac{\mu}{2})}  \,;
\]
\item $J_{\zeta,\epsilon} \pi_{\zeta,\epsilon}(x)
= \pi_{-\zeta,\epsilon}(x)  J_{\zeta,\epsilon}$ for all $x \in G$;
\item
the map $\zeta \mapsto J_{\zeta,\epsilon}$ extends analytically to the set $\{ \zeta \in \C : -\half < \Re\zeta < \half \}$, and (i) to (iii) continue to hold for these $\zeta$.
Further, $J_{-\zeta,\epsilon} J_{\zeta,\epsilon}$ is the identity map.
\end{enumerate}
\end{lemma}

\begin{proof}
Parts (i) and (iii) follow from the previous lemma and the definition of $J_{\zeta,\epsilon}$.

Next, from Lemma \ref{lem:intertwining-basis}, the definition of $c(\zeta,\epsilon)$, the duplication formula \eqref{eq:duplication}, and the reflection formula \eqref{eq:reflection}, applied twice, $J_{\zeta,\epsilon} f_{\zeta, \mu}$ is equal to $d(\zeta,\epsilon,\mu) f_{-\zeta,\mu}$, where
\begin{align*}
d(\zeta,\epsilon,\mu)
&=  \frac{b(\zeta,\mu)}{c(\zeta,\epsilon)} \\
&=  (-i)^\epsilon \sigma(\mu,\epsilon) \, 2^{1 + 4\zeta} \, \pi^{1/2} \frac{\Gamma(-2\zeta)\,\Gamma(\half + \frac{\epsilon}{2} +\zeta)}{\Gamma(\half - \zeta + \frac{\mu}{2}) \, \Gamma(\half - \zeta - \frac{\mu}{2})\, \Gamma(\frac{\epsilon}{2} - \zeta)} \\
&= \tilde\sigma(\mu,\epsilon)\,2^{2\zeta} \, \frac{\Gamma(-\zeta)\,\Gamma(\half - \zeta)\, \Gamma(\half + \frac{\epsilon}{2} +\zeta)}{\Gamma(\half - \zeta + \frac{\mu}{2}) \, \Gamma(\half - \zeta - \frac{\mu}{2})\, \Gamma(\frac{\epsilon}{2} - \zeta)}  \\
&= \tilde\sigma(\mu,\epsilon)\,2^{2\zeta} \, \frac{\Gamma(\half - \frac{\epsilon}{2} - \zeta)\, \Gamma(\half + \frac{\epsilon}{2} +\zeta)}{\Gamma(\half - \zeta + \frac{\mu}{2}) \, \Gamma(\half - \zeta - \frac{\mu}{2})}  \\
&= \tilde\sigma(\mu,\epsilon)\,2^{2\zeta} \, \frac{ \Gamma(\half + \zeta + \frac{\mu}{2}) \sin(\pi(\half - \zeta - \frac{\mu}{2}))} {\Gamma(\half - \zeta + \frac{\mu}{2})\sin(\pi(\half - \frac{\epsilon}{2} - \zeta))}   \,,
\end{align*}
where $\tilde\sigma(\mu,\epsilon) = (1 - \epsilon) \cos(\mu\pi/2) + \epsilon \sin(\mu\pi/2)$.
Unravelling this expression, first when $\epsilon = 0$ and then when $\epsilon = 1$, shows that
\begin{align*}
d(\zeta,\epsilon,\mu)
&= 2^{2\zeta}
\, \frac{ \Gamma(\half + \zeta + \frac{\mu}{2}) } {\Gamma(\half - \zeta + \frac{\mu}{2})}   \,,
\end{align*}
as claimed.
\end{proof}

Finally we are ready to introduce the generalised spherical functions.
Given integers $\mu,\nu$ in $2\Z + \epsilon$, we define
\begin{equation}\label{eq:def-gen-sph-functions}
\phi_{\zeta, \epsilon}^{\mu,\nu}(x)
=( \pi_{\zeta,\epsilon}(x) f_{\zeta,\mu}, f_{-\zeta,-\nu})
\qquad\forall x\in G,
\end{equation}
where $f_{\zeta,\mu}$ are the functions defined in equation~\eqref{eq:def-vettori}.

We recall that $\phi_{\zeta, 0}^{0,0}$ is the well-known bi-$K$--invariant spherical function, whose norm was computed in~\cite{Stp}, while the generalised spherical functions $\phi_{\zeta, \epsilon}^{\mu,\nu}$ were considered by Ehrenpreis and Mautner.
We are going to estimate the norm of all the $\phi_{\zeta, \epsilon}^{\mu,\nu}$; in the particular case where $\epsilon=\mu=\nu=0$, our result agrees with~\cite{Stp}.

We begin with some identities for the spherical functions.
First,
\[
(\phi_{\zeta,\epsilon}^{\mu,\nu})\bar{\phantom{x}}
= \phi_{\bar\zeta,\epsilon}^{-\mu,-\nu};
\]
this is proved by taking the complex conjugate of the definition.
Next,
\[
\begin{aligned}
(\phi_{\zeta,\epsilon}^{\mu,\nu})\check{\phantom{x}}(x)
&= ( \pi_{\zeta,\epsilon}(x^{-1}) f_{\zeta, \mu}, f_{-\zeta, -\nu} ) \\
&= ( f_{\zeta, \mu}, \pi_{-\zeta,\epsilon}(x)  f_{-\zeta, -\nu} ) \\
& = \phi_{-\zeta,\epsilon}^{-\nu, -\mu}(x)
\end{aligned}
\]
for all $x \in G$.
Finally, if $-\half < \Re\zeta < \half$, then
\begin{equation}\label{e:intertwining}
\begin{aligned}
( \pi_{\zeta, \epsilon}(x) f_{\zeta, \mu} , f_{-\zeta, -\nu} )
&= d(-\zeta, \epsilon, \mu)^{-1} ( \pi_{\zeta, \epsilon}(x) J_{-\zeta,\epsilon}
 f_{-\zeta, \mu} , f_{-\zeta, -\nu} ) \\
&= d(-\zeta, \epsilon, \mu)^{-1} ( J_{-\zeta,\epsilon} \pi_{-\zeta, \epsilon}(x) f_{-\zeta, \mu} , f_{-\zeta, -\nu} ) \\
&= d(-\zeta, \epsilon, \mu)^{-1} (-1)^\epsilon ( \pi_{-\zeta, \epsilon}(x) f_{-\zeta, \mu} , J_{-\zeta,\epsilon}  f_{-\zeta, -\nu} ) \\
&= (-1)^\epsilon \,\frac{d(-\zeta,\epsilon, \nu)}{d(-\zeta, \epsilon, \mu)} \, ( \pi_{-\zeta, \epsilon}(x) f_{-\zeta, \mu} ,  f_{\zeta, -\nu} ) \\
&= \frac{d(\zeta,\epsilon, \mu)}{d(\zeta, \epsilon, \nu)} \, ( \pi_{-\zeta, \epsilon}(x) f_{-\zeta, \mu} ,  f_{\zeta, -\nu} ),
\end{aligned}
\end{equation}
from \eqref{eq:gamma-ratios}.
In particular, it follows that $\phi^{\mu,\mu}_{-\zeta,\epsilon} = \phi^{\mu,\mu}_{\zeta,\epsilon}$.

In light of Corollary \ref{cor:ub-norm-identities}, these identities lead to some norm equalities for the uniformly bounded multiplier norms of the generalised spherical functions.

\section{Proof of Main Theorem}

In this section, we estimate the completely bounded multiplier norms of the spherical functions $\phi_{\zeta,\epsilon}^{\mu,\nu}$, and prove the Main Theorem.
The proof requires a number of steps that we separate out as lemmata, and other steps that are included in the text as explanation.

Throughout this section, we write $\zeta = \xi+ i \eta$, where $\xi, \eta \in \R$.

First of all, we may restrict to the case where $-\half < \xi < 0$.
Indeed, if $\xi =0$, then the generalised spherical functions are matrix coefficients of unitary representations and have $B(G)$ norm equal to $1$; \emph{a fortiori}, $\biglnorm\phi_{\zeta,\epsilon}^{\mu,\nu} \bigrnorm_{\cb} \leq 1$.
Next, if $\Re\zeta > 0$, then \eqref{e:intertwining} shows that $\phi_{\zeta,\epsilon}^{\mu,\nu}$ is a multiple of $(\phi_{-\zeta,\epsilon}^{\mu,\nu})$, and $\Re(-\zeta) < 0$; the result that we shall prove for the case where $\Re\zeta < 0$ will allow us to treat the general case.

Next, note that
\[
\phi_{\zeta, \epsilon}^{\mu,\nu}(k_{\theta} x k_{\psi})
= e^{i\nu\theta}\, e^{i\mu\psi}\, \phi_{\zeta, \epsilon}^{\mu,\nu}(x)
\qquad\forall x\in G \quad\forall\theta,\psi\in\R,
\]
therefore, by Proposition~\ref{uno-gen},
\[
\biglnorm \phi_{\zeta, \epsilon}^{\mu,\nu} \bigrnorm_{\cb}
= \biglnorm \phi_{\zeta, \epsilon}^{\mu,\nu} |_{\bar NA} \bigrnorm_{B}.
\]

To find the $B(\bar NA)$ norm, we are going to work on $\bar NA$, and this involves working with the noncompact version of the representations and the intertwining operators.
We write $\pi_{\zeta,\epsilon}^\R$ for the representation $\pi_{\zeta, \epsilon}$ acting on functions on $\R$, as in \eqref{eq:def-pi-R}.
Given a function $f$ on $\R^2 \setminus\{(0,0)\}$, we write $f^\R$ for the associated function on $\R$, that is, $f^\R(t) = f(1,t)$ for all $t \in \R$.
We write $(\cdot,\cdot)_\R$ for the standard pairing for functions on $\R$:
\[
(f, g)_\R = \int_\R f(t) \, g(t) \,dt;
\]
then for functions $f$ and $g$ on $\R^2 \setminus\{(0,0)\}$, it follows that
\[
(f, g) = (f^\R, g^\R)_\R .
\]
Finally, we write $J^\R_{\zeta,\epsilon}$ for the convolution operator on functions on $\R$ corresponding to Fourier multiplication by $\sgn^\epsilon(\cdot) \labs \cdot \rabs^{2\zeta}$; these are convolutions with kernels that are homogeneous of degree $-1-2\zeta$.
The transpose of the operator $J^\R_{\zeta,\epsilon}$ is $(-1)^\epsilon J^\R_{\zeta,\epsilon}$, for parity reasons.

Consistently with this notation, $J^\R_{\zeta/2,\epsilon}$ denotes the convolution operator on functions on $\R$ corresponding to Fourier multiplication by $\sgn^\epsilon(\cdot) \labs \cdot \rabs^{\zeta}$.
Consideration of the associated Fourier multipliers shows that
\[
(J^\R_{\zeta/2,0})^2 = J^\R_{\zeta,0},
\qquad
J^\R_{\zeta/2,\epsilon}J^\R_{\zeta/2,0} = J^\R_{\zeta,\epsilon},
\quad\text{and}\quad
(J^\R_{\zeta/2,\epsilon})^2 = J^\R_{\zeta,0}.
\]

It is now clear that
\[
\Biglpar \half(1-i) \biglpar J^\R_{\zeta/2,0} + i J^\R_{\zeta/2,\epsilon} \bigrpar \Bigrpar^2
= J^\R_{\zeta,\epsilon} ,
\]
so $\half(1-i) \biglpar J^\R_{\zeta/2,0} + i J^\R_{\zeta/2,\epsilon} \bigrpar$ is a square root of $J^\R_{\zeta,\epsilon}$, which we write $R _{\zeta/2,\epsilon}$; its transpose, written $S _{\zeta/2,\epsilon}$, is the operator $\half(1-i) \biglpar J^\R_{\zeta/2,0} + (-1)^\epsilon i J^\R_{\zeta/2,\epsilon} \bigrpar$, which is a square root of the transpose of $J^{\R}_{\zeta,\epsilon}$.

\begin{lemma}\label{lem:root-intertwining}
Suppose that $\epsilon \in \{ 0,1\}$ and that $-\half < \xi < 0$.
Then
\[
J^\R_{\zeta/2,\epsilon} \pi^\R_{\zeta,\epsilon} (x)
= \pi^\R_{0,\epsilon}(x) J^\R_{\zeta/2,\epsilon}
\qquad\text{and}\qquad
R_{\zeta/2,\epsilon} \pi^\R_{\zeta,\epsilon} (x)
= \pi^\R_{0,\epsilon}(x) R_{\zeta/2,\epsilon}
\]
for all $x \in \bar NA$.
\end{lemma}

\begin{proof}
As $R_{\zeta/2, \epsilon}$ is a linear combination of $J^\R_{\zeta/2, 0}$ and $J^\R_{\zeta/2, \epsilon}$, it suffices to show that both these operators have the desired intertwining property.

Consider equation \eqref{eq:def-pi-R}.
If $x \in \bar N$, then $\pi^\R_{\zeta,\epsilon}(x)$ acts on functions on $\R$ by translation, irrespective of the values of $\zeta$ and $\epsilon$; translation commutes with convolution, so what we need to prove is evident.

Assume now that $x = a_s \in A$.
Then
\begin{align*}
J^\R_{\zeta/2,\epsilon} \pi^\R_{\zeta,\epsilon}(x) f(s^2u)
&= \frac{1}{c(\zeta/2, \epsilon)} \int_\R \labs s^2u - t \rabs^{-1-\zeta} \sgn^\epsilon(s^2 u - t) \, \pi^\R_{\zeta,\epsilon/2}(x) f(t) \,dt \\
&= \frac{1}{c(\zeta/2, \epsilon)} \int_\R \labs s^2u - t \rabs^{-1-\zeta} \sgn^\epsilon(s^2 u - t) \, s^{2\zeta-1} f(s^{-2}t) \,dt \\
&= \frac{s^{2\zeta+1}}{c(\zeta'2, \epsilon)} \int_\R \labs s^2 u - s^2 t \rabs^{-1-\zeta} \sgn^\epsilon(s^2 u - s^2 t) \,  f(t) \,dt \\
&= \frac{s^{-1}}{c(\zeta/2, \epsilon)} \int_\R \labs u - t \rabs^{-1-\zeta} \sgn^\epsilon(u - t) \,  f(t) \,dt \\
&= s^{-1} J_{\zeta/2,\epsilon}  f(u) ,
\end{align*}
whence
\[
J^\R_{\zeta/2,\epsilon} \pi^\R_{\zeta,\epsilon}(x) f(u)
= \pi^\R_{0,\epsilon}(x) J^\R_{\zeta/2,\epsilon} f(u)
\qquad\forall u \in \R.
\]
Since the desired result holds when $x \in \bar N$ and when $x \in A$, it holds for all $x \in \bar NA$.
\end{proof}

Note that
\begin{align*}
\biglnorm R_{\zeta/2, \epsilon} f \bigrnorm_2
= \frac{1}{(2\pi)^{1/2}} \, \biglnorm (R_{\zeta/2, \epsilon} f)\hat{\phantom{x}} \bigrnorm_2
= \frac{1}{(2\pi)^{1/2}} \, \biglnorm \labs\cdot\rabs^{\xi} \hat f \bigrnorm_2
= \biglnorm J^{\R}_{\xi/2, 0} f \bigrnorm_2  \\
\noalign{\noindent{and}}
\biglnorm S_{\zeta/2, \epsilon} f \bigrnorm_2
= \frac{1}{(2\pi)^{1/2}} \, \biglnorm (S_{\zeta/2, \epsilon} f)\hat{\phantom{x}} \bigrnorm_2
= \frac{1}{(2\pi)^{1/2}} \, \biglnorm \labs\cdot\rabs^{\xi} \hat f \bigrnorm_2
= \biglnorm J^{\R}_{\xi/2, 0} f \bigrnorm_2
\end{align*}
for all functions $f$ on $\R$ for which the last term is finite, by the Plancherel theorem; here $\biglnorm \cdot \bigrnorm_2$ indicates the standard $L^2(\R)$ norm.

For all $x \in \bar NA$,
\begin{equation}\label{eq:sph-fun-reg-repn}
\begin{aligned}
\phi_{\zeta,\epsilon}^{\mu,\nu}(x)
& = ( \pi_{\zeta, \epsilon}(x) f_{\zeta,\mu} , f_{-\zeta, -\nu} ) \\
& = d(\zeta,\epsilon, -\nu)^{-1} ( \pi_{\zeta,\epsilon}(x) f_{\zeta,\mu} , J_{\zeta, \epsilon} f_{\zeta, -\nu} ) \\
& = d(\zeta,\epsilon, -\nu)^{-1} (-1)^\epsilon ( J_{\zeta, \epsilon} \pi_{\zeta,\epsilon}(x) f_{\zeta,\mu} , f_{\zeta, -\nu} ) \\
& = d(\zeta,\epsilon, -\nu)^{-1} (-1)^\epsilon ( R_{\zeta/2, \epsilon} R_{\zeta/2, \epsilon} \pi^\R_{\zeta, \epsilon}(x) f_{\zeta,\mu}^\R , f_{\zeta, -\nu}^\R )_\R \\
& = d(\zeta,\epsilon, -\nu)^{-1} (-1)^\epsilon ( \pi^\R_{0, \epsilon}(x) R_{\zeta/2, \epsilon}  f_{\zeta,\mu}^\R , S_{\zeta/2, \epsilon} f_{\zeta, -\nu}^\R )_\R  .
\end{aligned}
\end{equation}
Now $\pi_{0,\epsilon}^\R$ is a unitary representation on $L^2(\R)$, and so
\begin{equation}\label{eq:first-expression}
\begin{aligned}
\biglnorm \phi_{\zeta,\epsilon}^{\mu,\nu} \bigrnorm_{\cb}
&\leq \labs d(\zeta,\epsilon, -\nu) \rabs^{-1}\biglnorm R_{\zeta/2, \epsilon}  f_{\zeta,\mu}^\R\bigrnorm_{2} \biglnorm S_{\zeta/2, \epsilon}  f_{\zeta,-\nu}^\R\bigrnorm_{2} \\
&= \labs d(\zeta,\epsilon, -\nu) \rabs^{-1} \biglnorm J^\R_{\xi/2, 0}  f_{\zeta,\mu}^\R\bigrnorm_{2} \biglnorm J^\R_{\xi/2, 0}  f_{\zeta,-\nu}^\R\bigrnorm_{2} .
\end{aligned}
\end{equation}

When $\mu = \nu$, more may be said.
More precisely, from the observation following \eqref{eq:basis-on-R}, $\hat f^{\,\R}_{\zeta, \mu} = \biglpar \hat f^{\,\R}_{\zeta, -\mu} \bigrpar \check{\phantom{k}}$, so
\[
\biglabs \cF \biglpar R_{\zeta/2,\epsilon} \, f^\R_{\zeta,\mu} \bigrpar\bigrabs
= \biglabs \biglpar \cF \biglpar S_{\zeta/2,\epsilon} f^\R_{\zeta,-\mu} \bigrpar\bigrpar \check{\phantom{k}} \bigrabs.
\]
Hence  from \eqref{eq:sph-fun-reg-repn} and \eqref{eq:abelian-A-norm},
\[
\begin{aligned}
\biglnorm \phi_{\zeta,\epsilon}^{\mu,\nu} \bigrnorm_{\cb}
&= \biglnorm \phi_{\zeta,\epsilon}^{\mu,\nu} \bigm|_{\bar NA} \bigrnorm_{B(\bar NA)}
 \geq \biglnorm \phi_{\zeta,\epsilon}^{\mu,\nu} \bigm|_{\bar N} \bigrnorm_{A(\bar N)} \\
&= \biglnorm \cF( \phi_{\zeta,\epsilon}^{\mu,\nu} \bigm|_{\bar N} ) \bigrnorm_{1}
= \labs d(\zeta,\epsilon, -\nu) \rabs^{-1} \biglnorm J^\R_{\xi/2, 0}  f_{\zeta,\mu}^\R\bigrnorm_{2} \biglnorm J^\R_{\xi/2, 0}  f_{\zeta,-\nu}^\R\bigrnorm_{2} ,
\end{aligned}
\]
and equality holds in \eqref{eq:first-expression}.

The proof of the theorem now hinges on the estimation of these norms, which is the subject of the next lemma.
First, by considering conjugation and reflection, we see that
\[
\biglnorm J^\R_{\xi/2, 0}  f_{\zeta,\nu}^\R\bigrnorm_{2}
= \biglnorm J^\R_{\xi/2, 0}  f_{\zeta,-\nu}^\R\bigrnorm_{2}
= \biglnorm J^\R_{\xi/2, 0}  f_{\bar\zeta,\nu}^\R\bigrnorm_{2}.
\]

\begin{lemma} \label{calcolo}
Suppose that $-\half < \xi < 0$ and $\mu$ is an integer.
Then $J^\R_{\xi/2, 0}  f_{\zeta,\mu}^\R$ is in $L^2(\R)$ and $\biglnorm J^\R_{\xi/2, 0}  f_{\zeta,\mu}^\R\bigrnorm_{2}^2$ is equal to
\[ 2^{2\xi}\,
\Re\lpar
\frac{
 \Gamma(\half + \bar\zeta + \frac{\mu}{2}) }
{\Gamma(\half - \zeta + \frac{\mu}{2}) } \,
( 1  + i  \tan(\pi\xi) \tanh(\pi \eta) ) \rpar
\]
when $\mu$ is even, and to
\[
2^{2\xi}\,
\Re\lpar
\frac{
 \Gamma(\half + \bar\zeta + \frac{\mu}{2}) }
{\Gamma(\half - \zeta + \frac{\mu}{2}) } \,
( 1 + i \tan(\pi\xi) \coth(\pi\eta) ) \rpar
\]
when $\mu$ is odd.
\end{lemma}

\begin{proof}
By the Plancherel formula,  setting $t = \tan\theta$, $u=\tan\phi$, and changing variables as in \eqref{eq:double-integral-invariance}, we obtain
\begin{align*}
\biglnorm  J^\R_{\xi/2, 0}  f_{\zeta,\mu}^\R  \bigrnorm_2^2
&=
\lip J^\R_{\xi/2, 0}  f_{\zeta,\mu}^\R , J^\R_{\xi/2, 0}  f_{\zeta,\mu}^\R \rip \\
&= \lip J^\R_{\xi, 0}  f_{\zeta,\mu}^\R , f_{\zeta,\mu}^\R \rip \\
&= ( J^\R_{\xi, 0}  f_{\zeta,\mu}^\R , \bar f_{\zeta,\mu}^\R )_\R \\
&= \frac{1}{c(\xi,0)} ( I^\R_{\xi, 0}  f_{\zeta,\mu}^\R , f_{\bar\zeta,-\mu}^\R )_\R \\
&= \frac{1}{c(\xi,0)}
\int_{\R} \int_{\R}
f_{\zeta, \mu}(1,t)\, f_{\bar\zeta, -\mu}(1,u)\,\labs t-u \rabs^{-2\xi-1}\, dt\, du
\\
&= \frac{1}{\pi \, c(\xi,0)} \int_{-\pi/2}^{\pi/2} \int_{-\pi/2}^{\pi/2}
e^{i\mu(\theta-\varphi)}\,
\Bigl(\frac{1+\tan^2\theta}{1+\tan^2\varphi}\Bigr)^{i\eta}\,
\labs \cosec(\theta-\varphi) \rabs^{1+2\xi}\, d\theta\, d\varphi ;
\end{align*}
the inner product here is the $L^2$ inner product on functions on $\R$.

Define the functions $m_{\epsilon,\mu,\eta}$ and $h_{\xi}$ on the unit circle by the formulae.
\begin{gather*}
m_{\epsilon,\mu,\eta}(\cos\theta,\sin\theta)
=\sgn^\epsilon(\cos\theta)\,e^{i\mu\theta}\labs\cos\theta\rabs^{-2i\eta} \\
h_{\xi}(\cos\theta,\sin\theta)=\labs \cosec\theta \rabs^{1+2\xi},
\end{gather*}
where $\theta\in \R$.
Then $m_{\epsilon,\mu,\eta}$ and $h_{\xi}$ are $\pi$-periodic, so that
\[
\biglnorm  J^\R_{\xi/2, 0}  f_{\zeta,\mu}^\R  \bigrnorm_2^2
=
\frac{1}{4 \, \pi \, c(\xi,0)}
\int_{-\pi}^{\pi}
\int_{-\pi}^{\pi}
m_{\epsilon,\mu,\eta}(\theta)\,{h_{\xi}}(\theta-\varphi)\,\bar{m}_{\epsilon,\mu,\eta}(\varphi)
\, d\varphi\,
 d\theta,
\]
and the last double integral may be interpreted as the inner product on the circle between the functions $m_{\epsilon,\mu,\eta}$ and $h_{\xi} * m_{\epsilon,\mu,\eta}$.
Therefore, by the Parseval identity,
\begin{align*}
\biglnorm J^\R_{\xi/2, 0}  f_{\zeta,\mu}^\R\bigrnorm_{2}^2
&=\frac{\pi}{c(\xi,0)} \,\sum_{\kappa \in \Z} \labs \tilde{m}_{\epsilon,\mu,\eta}(\kappa ) \rabs^2\, \tilde{h}_{\xi}(\kappa ),
\end{align*}
where the tilde indicates the Fourier coefficients, which we now compute.

By Lemma \ref{lem:definite-integral},
\begin{align*}
\tilde{h}_{\xi}(\kappa )
&=\frac1{2\pi}
\int_{-\pi}^{\pi}
\labs \sin\theta \rabs^{-2\xi-1}
\,e^{-i\kappa \theta}\,d\theta,
\\
&=\frac1{\pi}
\int_{0}^{\pi}
|\sin(\theta)|^{-2\xi-1}
\,\cos(\kappa \theta)\,d\theta
\\
&=2^{1+2\xi}\cos(\kappa \pi/2)
\frac{\Gamma(-2\xi)}{\Gamma(\half - \xi + \frac{\kappa }{2})\, \Gamma(\half - \xi - \frac{\kappa }{2})}
\end{align*}
which is $0$ unless $\kappa $ is even; we assume this for the rest of this computation.
Thus from the duplication formula \eqref{eq:duplication} and the reflection formula \eqref{eq:reflection}, applied twice,
\begin{align*}
\tilde{h}_{\xi}(\kappa )
&=\cos(\kappa \pi/2) \, \pi^{-1/2}\,
\frac{\Gamma(-\xi) \, \Gamma(\half - \xi)}
{\Gamma(\half - \xi + \frac{\kappa }{2}) \, \Gamma(\half - \xi - \frac{\kappa }{2})} \\
&=\cos(\kappa \pi/2) \, \pi^{-1/2}\,
\frac{\Gamma(-\xi)}{\Gamma(\half - \xi + \frac{\kappa }{2})} \, \frac{\Gamma(\half - \xi)}{\Gamma(\half - \xi - \frac{\kappa }{2})} \\
&=\cos(\kappa \pi/2) \, \pi^{-1/2}\,
\frac{\Gamma(-\xi)}{\Gamma(\half - \xi + \frac{\kappa }{2})} \, \frac{\sin(\pi(\half + \xi + \frac{\kappa }{2})) \, \Gamma(\half + \xi + \frac{\kappa }{2})}{\sin(\pi( \half + \xi)) \, \Gamma(\half + \xi)} \\
&=\cos^2(\kappa \pi/2)\, \pi^{-1/2}\,
\frac{\Gamma(-\xi)}{\Gamma(\half + \xi)}
\frac{\Gamma(\half + \xi + \frac{\kappa }{2})}{\Gamma(\half - \xi + \frac{\kappa }{2})} \,.
\end{align*}

It is enough to evaluate $\tilde{m}_{\epsilon,\mu,\eta}(\kappa )$ when $\kappa $ is even.
The case where $\eta=0$ is easier, and we suppose that $\eta\neq0$.
By parity, a change of variable, formula \eqref{eq:integral-formula}, and the evenness of $\epsilon +\mu$ and of $\kappa $,
\begin{align*}
\tilde{m}_{\epsilon,\mu,\eta}(\kappa )
&=
\frac1{2\pi}
\int_{-\pi}^{\pi}
\labs\cos\theta\rabs^{-2i\eta}\,e^{-i(\kappa -\mu)\theta} \sgn^\epsilon(\cos\theta) \,d\theta
\\
&=
\frac{1}{\pi}
\int_{0}^{\pi}
\labs\cos\theta\rabs^{-2i\eta} \cos((\kappa -\mu)\theta) \sgn^\epsilon(\cos\theta) \,d\theta
\\
&=
\frac{1}{\pi}
\int_{0}^{\pi/2}
\labs\cos\theta\rabs^{-2i\eta} \cos((\kappa -\mu)\theta) \,d\theta \\
&\qquad +
\frac{(-1)^\epsilon}{\pi}
\int_{\pi/2}^{\pi}
\labs\cos\theta\rabs^{-2i\eta} \cos((\kappa -\mu)\theta) \,d\theta
\\
&=
\frac{1}{\pi}
\int_{0}^{\pi/2}
\labs\cos\theta\rabs^{-2i\eta} \cos((\kappa -\mu)\theta) \,d\theta \\
&\qquad +
\frac{(-1)^\epsilon}{\pi}
\int_{0}^{\pi/2}
\labs\cos(\pi-\theta)\rabs^{-2i\eta} \cos((\kappa -\mu)(\pi-\theta)) \,d\theta
\\
&=
\frac{1}{\pi}
\int_{0}^{\pi/2}
\labs\cos\theta\rabs^{-2i\eta} \cos((\kappa -\mu)\theta) \,d\theta \\
&\qquad +
\frac{(-1)^\epsilon \cos((\kappa -\mu)\pi)}{\pi}
\int_{0}^{\pi/2}
\labs\cos\theta\rabs^{-2i\eta} \cos((\kappa -\mu)\theta) \,d\theta
\\
&=
\frac{2}{\pi}
\int_{0}^{\pi/2}
\labs\cos\theta\rabs^{-2i\eta} \cos((\kappa -\mu)\theta) \,d\theta \\
&= 2^{2i\eta} \frac{\Gamma(1 - 2i\eta)}{\Gamma(1 - i\eta + \frac{\kappa }{2} - \frac{\mu}{2} )\,\Gamma(1 - i\eta  - \frac{\kappa }{2} + \frac{\mu}{2}) } \,.
\end{align*}
We now apply the duplication formula \eqref{eq:duplication} to the gamma function in the numerator, and the reflection formula \eqref{eq:reflection} to a gamma function in the denominator to deduce that
\begin{align*}
\tilde{m}_{\epsilon,\mu,\eta}(\kappa )
&= \pi^{-3/2} \frac{\Gamma(\half - i\eta) \, \Gamma(1 - i\eta) \,\Gamma(i\eta  + \frac{\kappa }{2} - \frac{\mu}{2})} {\Gamma(1 - i\eta + \frac{\kappa }{2} - \frac{\mu}{2} ) } \sin(\pi(i\eta  + \frac{\kappa }{2} - \frac{\mu}{2})) \,.
\end{align*}

Note that, since $\Gamma(\bar z) = \bar\Gamma(z)$, by the reflection formula,
\[
\labs \Gamma(\half-i\eta) \rabs^2
= \Gamma(\half-i\eta)\,\Gamma(\half+i\eta)
= \Gamma(1 - \half - i\eta) \, \Gamma(\half+i\eta)
= \frac{\pi}{\sin(\pi(\half+i\eta)) } \,,
\]
and similarly
\[
\labs \Gamma(1-i\eta) \rabs^2
= i\eta \,\Gamma(i\eta)\,\Gamma(1-i\eta)
= \frac{i\eta \pi}{\sin(\pi(i\eta)) } \,,
\]
while
\[
\labs \frac{\Gamma(i\eta  + \frac{\kappa }{2} - \frac{\mu}{2})} {\Gamma(1 - i\eta + \frac{\kappa }{2} - \frac{\mu}{2} ) } \rabs^2
= \labs \frac{\Gamma(i\eta  + \frac{\kappa }{2} - \frac{\mu}{2})} {\Gamma(1 + i\eta + \frac{\kappa }{2} - \frac{\mu}{2} ) } \rabs^2
= \frac{1}{\eta^2 + \biglpar \frac{\kappa }{2} - \frac{\mu}{2} \bigrpar^2} \,.
\]
Thus
\begin{align*}
\labs \tilde{m}_{\epsilon,\mu,\eta}(\kappa ) \rabs^2
& = \frac{1}{\pi} \,
\frac{ \labs \sin(\pi(i\eta - \frac{\mu}{2})) \rabs^2 }
{ \sinh(\pi\eta) \cosh(\pi \eta) } \,
\frac{\eta }{ \eta^2 + \biglpar \frac{\kappa }{2} - \frac{\mu}{2} \bigrpar^2} \\
& = \frac{1}{\pi} \,
\gamma(\eta,\mu) \,
\frac{\eta }{ \eta^2 + \biglpar \frac{\kappa }{2} - \frac{\mu}{2} \bigrpar^2} \,,
\end{align*}
say.

Hence when $\eta\neq 0$,
\begin{align*}
\biglnorm J^\R_{\xi/2, 0}  f_{\zeta,\mu}^\R\bigrnorm_{2}^2
&=\frac{2^{2\xi}}{\pi} \,
\gamma(\eta,\mu)\,\sum_{\kappa \in \Z}
\frac{\eta}{\eta^2+\frac{(\kappa -\mu)^2}{4}}\,
\cos^2(\kappa \pi/2)\,
\frac{\Gamma(\half + \xi + \frac{\kappa }{2})}{\Gamma(\half - \xi + \frac{\kappa }{2})}
\\
&=\frac{2^{2\xi}}{\pi} \,
\gamma(\eta,\mu)\,
\sum_{\kappa \in \Z}
\frac{\eta}{\eta^2+(\kappa -\frac{\mu}{2})^2}\,
\frac{\Gamma(\half + \xi + \kappa )}{\Gamma(\half - \xi + \kappa )} \,.
\end{align*}
To evaluate the sum, we note that
\begin{align*}
\frac{\eta}{(\kappa -\frac{\mu}{2} )^2+\eta^2}
=-\Im\lpar
\frac{\Gamma(i\eta + \kappa - \frac{\mu}{2})}{\Gamma(1 + i\eta + \kappa - \frac{\mu}{2})} \rpar,
\end{align*}
and apply Dougall's formula \eqref{eq:Dougall} and then the reflection formula \eqref{eq:reflection} to deduce that
\begin{align*}
\biglnorm J^\R_{\xi/2, 0}  f_{\zeta,\mu}^\R\bigrnorm_{2}^2
&=-\frac{2^{2\xi}}{\pi} \,
\gamma(\eta,\mu)\,\Im\lpar
\sum_{\kappa \in \Z}
\frac{\Gamma(\half + \xi + \kappa )}{\Gamma(\half - \xi + \kappa )}
\frac{\Gamma(i\eta + \kappa - \frac{\mu}{2})}{\Gamma(1 + i\eta  + \kappa -\frac{\mu}{2})} \rpar
\\
&=
-2^{2\xi}\, \pi\,
\gamma(\eta,\mu)\,
\Im\lpar
\frac{
\cosec(\pi(\half + \xi))\,\cosec(\pi(i\eta - \frac{\mu}{2} ))}
{\Gamma(\half - \xi + i\eta -\frac{\mu}{2})\,\Gamma(\half - \xi - i\eta + \frac{\mu}{2})} \rpar
\\
&=
-2^{2\xi}\,\gamma(\eta,\mu)\,
\Im\lpar
\frac{
 \Gamma(\half + \bar\zeta + \frac{\mu}{2}) \sin(\pi(\half + \bar\zeta + \frac{\mu}{2}))}
{\Gamma(\half - \zeta + \frac{\mu}{2}) \sin(\pi(i\eta-\frac{\mu}{2})) \cos(\pi\xi)} \rpar \\
&=
-2^{2\xi}\, \frac{ \labs \sin(\pi(i\eta - \frac{\mu}{2})) \rabs^2 }
{ \sinh(\pi\eta) \cosh(\pi \eta) } \,
\Im\lpar
\frac{
 \Gamma(\half + \bar\zeta + \frac{\mu}{2}) \cos(\pi(\bar\zeta + \frac{\mu}{2}))}
{\Gamma(\half - \zeta + \frac{\mu}{2}) \sin(\pi(i\eta - \frac{\mu}{2})) \cos(\pi\xi)} \rpar \\
&=
2^{2\xi}\,
\Im\lpar
\frac{
 \Gamma(\half + \bar\zeta + \frac{\mu}{2}) }
{\Gamma(\half - \zeta + \frac{\mu}{2}) } \,
\frac{ \sin(\pi(i\eta + \frac{\mu}{2})) \cos(\pi(\bar\zeta + \frac{\mu}{2}))}
{ \sinh(\pi\eta) \cosh(\pi \eta)  \cos(\pi\xi)} \rpar,
\end{align*}
which is equal to
\[ 2^{2\xi}\,
\Re\lpar
\frac{
 \Gamma(\half + \bar\zeta + \frac{\mu}{2}) }
{\Gamma(\half - \zeta + \frac{\mu}{2}) } \,
( 1  + i  \tan(\pi\xi) \tanh(\pi \eta) ) \rpar
\]
when $\mu$ is even, and to
\begin{equation}\label{bad-estimate}
   2^{2\xi}\,
\Re\lpar
\frac{
 \Gamma(\half + \bar\zeta + \frac{\mu}{2}) }
{\Gamma(\half - \zeta + \frac{\mu}{2}) } \,
( 1 + i \tan(\pi\xi) \coth(\pi\eta) ) \rpar
\end{equation}
when $\mu$ is odd.
\end{proof}

We now have all the ingredients to prove the theorem.
We first consider the case where $\epsilon = 0$.
Observe that
\begin{align*}
\biglnorm J^\R_{\xi/2, 0}  f_{\zeta,\mu}^\R\bigrnorm_{2}^2
&= \labs 2^{2\xi}\,
\Re\lpar
\frac{
 \Gamma(\half + \bar\zeta + \frac{\mu}{2}) }
{\Gamma(\half - \zeta + \frac{\mu}{2}) } \,
( 1 + i  \tan(\pi\xi) \tanh(\pi \eta) ) \rpar \rabs \\
&\leq 2^{2\xi}\,
\labs \frac{
 \Gamma(\half + \bar\zeta + \frac{\mu}{2}) }
{\Gamma(\half - \zeta + \frac{\mu}{2}) } \rabs
\labs 1 + i  \tan(\pi\xi) \tanh(\pi \eta) \rabs
\\
&\leq 2^{2\xi} \sec(\pi\xi)
\labs \frac{
 \Gamma(\half + \zeta + \frac{\mu}{2}) }
{\Gamma(\half - \zeta + \frac{\mu}{2}) } \rabs .
\end{align*}
Hence from \eqref{eq:first-expression} and \eqref{eq:gamma-ratios},
\begin{equation}\label{none}
\begin{aligned}
\biglnorm \phi_{\zeta,\epsilon}^{\mu,\nu} \bigrnorm_{\cb}
&\leq \sec(\pi\xi) \labs \frac {\Gamma(\half + \frac{\nu}{2} - \zeta )} { \Gamma(\half + \frac{\nu}{2} + \zeta ) }\rabs^{1/2}
\labs \frac{
 \Gamma(\half + \frac{\mu}{2} + \zeta) }
{\Gamma(\half + \frac{\mu}{2} - \zeta) } \rabs ^{1/2}.
\end{aligned}
\end{equation}
When $\mu = \nu$, the two ratios of gamma functions cancel, and we obtain the estimate
\[
\begin{aligned}
\biglnorm \phi_{\zeta,\epsilon}^{\mu,\nu} \bigrnorm_{\cb}
&\leq \sec(\pi\xi) ,
\end{aligned}
\]
which is very similar to Steenstrup's estimate \cite{Stp}.
In general, from \eqref{eq:gamma-ratios-limit}, we may write
\[
\biglnorm \phi_{\zeta,\epsilon}^{\mu,\nu} \bigrnorm_{\cb}
\leq C_{\xi_0} \, \Biglpar\frac{ |\mu| + 1 }{ | \nu| + 1 } \Bigrpar^\xi,
\]
uniformly for $\zeta$ such that $| \Re\zeta | \leq \xi_0 < \half$.
In particular, we can bound the norms of the spherical functions uniformly when $\xi > 0$ and  $\labs \mu \rabs \leq \labs \nu \rabs$ and when $\xi < 0$ and  $\labs \mu \rabs \geq \labs \nu \rabs$.

Now we consider the case where $\epsilon = 1$.
In this case, a similar argument shows that
\begin{equation}\label{one}
\begin{aligned}
\biglnorm \phi_{\zeta,\epsilon}^{\mu,\nu} \bigrnorm_{\cb}
&\leq \sec(\pi\xi) \coth(\pi\eta)
\labs \frac {\Gamma(\half + \frac{\nu}{2} - \zeta )} { \Gamma(\half + \frac{\nu}{2} + \zeta ) }\rabs^{1/2}
\labs \frac{\Gamma(\half + \frac{\mu}{2} + \zeta) }
{\Gamma(\half + \frac{\mu}{2} - \zeta) } \rabs ^{1/2}.
\end{aligned}
\end{equation}
Again when $\mu = \nu$, there is cancellation of the gamma factors.
The difficulty is that the hyperbolic cotangent becomes infinite as $\eta \to 0$.
We can show that the limit as $\eta$ approaches $0$ of the expression \eqref{bad-estimate} is finite, but we have not computed it exactly.
Consequently, we may again assert that
\[
\sup_{\eta\in\R} \biglnorm \phi_{\zeta,\epsilon}^{\mu,\nu} \bigrnorm_{\cb} < \infty
\]
for all $\mu$ and $\eta$, which again suggests that a sharper inequality for the norm should exist.

\section{Acknowledgements}
The first and third named authors were supported by GNAMPA; the second named author was supported by the Australian Research Council.

\end{document}